\documentclass[10pt]{article}
\usepackage{graphicx}
\usepackage{amsmath}
\usepackage{amsfonts}
\usepackage{amssymb}
\usepackage{amsthm}
\usepackage{verbatim}
\renewcommand{\-}{-\!\!\!\!\!\!\!\hspace{0.4mm}}

\newcommand{\lap}{\Delta}
\newcommand{\e}{\varepsilon}
\newcommand{\R}{\mathbb{R}}
\newcommand{\dd}{\partial}
\renewcommand{\d}{d}
\newcommand{\D}{\tilde D}
\newcommand{\dist}{\operatorname{dist}}

\begin{document}
\renewcommand{\rmdefault}{cmr}
\newtheorem{thm}{Theorem}
\newtheorem{lem}[thm]{Lemma}
\newtheorem{cor}[thm]{Corollary}
\newtheorem{prop}[thm]{Proposition}
\newtheorem{definition}[thm]{Definition}
\newtheorem{rem}[thm]{Remark}

\title{Optimal Regularity for the parabolic No-Sign Obstacle Problem}

\author{John Andersson, Erik Lindgren and Henrik Shahgholian}

\maketitle
\begin{abstract}\noindent
We study the parabolic free boundary problem of obstacle type
$$
\lap  u-\frac{\partial u}{\partial t}= f\chi_{\{u\ne 0\}}.
$$
Under the condition that
$
f=Hv
$
for some function $v$ with bounded second order spatial derivatives and bounded first order time derivative, we establish the same regularity for the solution $u$. Both the regularity and the assumptions are optimal.

Using this result and assuming that $f$ is Dini continuous, we prove that the free boundary is, near so called low energy points, a $C^1$ graph.

Our result completes the theory for this type of problems for the heat operator.
\end{abstract}

\section{Introduction}
We present a proof of the interior optimal regularity, that is, $W^{2,1}_\infty$-regularity (bounded second order spatial derivatives and bounded first order time derivative) for solutions to 
the so-called parabolic no-sign obstacle problem 
\begin{equation}\label{main}
\left\{\begin{array}{ll}
Hu :=\lap  u-\frac{\partial u}{\partial t}= f(x,t)\chi_{\{u\ne 0\}} & \textrm{ in } Q_1^-, \\
u=g & \textrm{ on }\partial_p Q_1^-,
\end{array}\right.
\end{equation}
under minimal assumptions on the data. 
Here $B_r$ is the unit ball, $Q_r^-=B_r\times (-r^2,0]$, $f=Hv$ where $v\in W^{2,1}_\infty$. Since we are interested in the interior regularity,
the assumptions on $g$ are not very important but we assume that $g$ is bounded and 
continuous for the sake of definiteness.

Naturally, we cannot expect that a solution $u$ of (\ref{main}) is in any better regularity class
than the solution, $v$, of the heat equation with right hand side $f$
$$Hv=f.
$$
Noticing that if $f\in L^p$ then $f\chi_{\{u\ne 0\}}\in L^p$ so for $f$ in the $L^p$-range Calderon-Zygmund
theory directly implies that $u\in W^{2,1}_p$ - which is as good regularity as $v$.

It is also easy to see that $u$ is no better than $C^{1,1}-$regular (bounded second order spatial derivatives) in the space directions 
even in the case when $f$ is constant. As a matter of fact, with an appropriate choice of $g$ the function
$$
u(x,t)=\frac{1}{2}(\textrm{max}(0,x_1))^2
$$
is a solution with $f(x,t)=1$. Similarly, the example
$$
u(x,t)=-(t-1/2)^+
$$
shows that $u(x,t)$ is no better than $C^{0,1}$ regular (bounded first order time derivative) in time. This shows that even for
constant $f$ we cannot hope for better regularity than $u\in W^{2,1}_\infty$. This leads to 
the question: What is the weakest possible assumptions on $f$ that assures that a solution
$u(x,t)$ of (\ref{main}) satisfies $u\in W^{2,1}_\infty$?

Since, in general, a solution to (\ref{main}) is never more regular than a solution
to $Hv=f$ one might ask: If $f$ is such that the solution $v$ to $Hv=f$ satisfies
$v\in W^{2,1}_\infty$ will a solution to (\ref{main}) also satisfy $u\in W^{2,1}_\infty$?

In the main theorem of this paper we answer the question in the affirmative. This
clearly provides an optimal regularity result for the no-sign parabolic
obstacle problem. Our main theorem is the following.

\begin{thm}\label{main1} 
Let $u$ be a solution to (\ref{main}) and assume furthermore that
$f= Hv$ in $Q^-_1$ where $v\in W^{2,1}_\infty(Q_1^-)$ and that $u\in L^1(Q_1^-)$.
Then
$u\in W^{2,1}_\infty(Q_{1/2}^-)$ and
\begin{equation}\label{EstinMainThm}
\|D^2 u\|_{L^\infty(Q_{1/2}^-)}+\|u_t\|_{L^\infty(Q_{1/2}^-)}\le C\big( \|u\|_{L^1(Q_1^-)}+\|D^2 v\|_{L^\infty(Q_1^-)}+\|v_t\|_{L^\infty(Q_{1}^-)}\big),
\end{equation}
where $C$ depends on the dimension.
\end{thm}

It might be in order to comment on the assumption that $Hv=f$ where $v\in W^{2,1}_\infty(Q_1^-)$.
The assumption might look technical, but to the authors' knowledge there is no necessary condition on
$f$ that assures that the solution $v$ to $Hv=f$ is in $W^{2,1}_\infty$. The weakest
sufficient condition known to the authors is $f\in C^{\text{Dini}}(Q_1^-)$ (see Section \ref{sec:notation} and for instance \cite{ME71} or \cite{w06}). Naturally, the theorem
holds, with minor changes in (\ref{EstinMainThm}), under the assumption that $f\in C^{\text{Dini}}(Q_1^-)$.
But the assumption we make on $f$ is strictly weaker than Dini-continuity and it also
highlights that up to $W^{2,1}_\infty$ the regularity of solutions to the no-sign parabolic
obstacle problem are as regular as the solutions to the corresponding heat equation.

Before we give a brief sketch of the history of the problem and formulate our second main theorem 
we would  
like to remark that this is a free boundary result. As a matter of fact the problem can be 
formulated as follows. Given $f\in L^p(Q_1^-)$ and $g\in C(\partial_p Q_1^-)$ find a set
$\Omega\subset Q_1^-$ and a function $u\in W^{2,1}_p(Q_1^-)$ solving
$$
\left\{\begin{array}{ll}
Hu= f\chi_{\Omega} & \textrm{ in } Q_1^-, \\
u=|\nabla u|=0 &\text{ in } Q_1^-\setminus \Omega, \\
u=g & \textrm{ on }\partial_p Q_1^-.
\end{array}\right.
$$
The free boundary is $\partial_p \textrm{interior}(\{u=0\}) \cap Q_1^-= \partial_p \{u\ne 0\}\cap Q_1^-$.
In our proof we will use  that $u=0$ outside of spt$(Hu)$ (to be more precise, we use that $\|D^2 u\|=u_t=0$ a.e.). The above theorem is not true, in general, for solutions to
\begin{equation}\label{potentialtheoryform}
\left\{\begin{array}{ll}
Hu= f\chi_{\Sigma} & \textrm{ in } Q_1^-, \\
u=g & \textrm{ on }\partial_p Q_1^-,
\end{array}\right.
\end{equation}
where $\Sigma$ is an arbitrary set. A simple counterexample for the time independent case, with
$\Sigma=\{x_1 x_2>0\}$, is given in \cite{ASW}. In \cite{ASW} a solution, $u$, to 
(\ref{potentialtheoryform}) with $f(x,t)=-1$ and $\Sigma=\{x_1 x_2>0\}$ is explicitly
calculated and $u(x,t)=u(x,0)\ne C^{1,1}(B_1(0))$. That $f(x,t)<0$ in \cite{ASW}, whereas we usually
think of $f>0$ in obstacle problems is irrelevant since (\ref{potentialtheoryform}) is
linear, so we can simply change the sign of the equation by taking $-u$ instead of $u$. The moral sense of the example is that the solution to the free boundary problem
choses a zero level set that assures that the solution has slightly better regularity
than what one in general would expect a solution to the similar problem (\ref{potentialtheoryform})
to have.

Before we continue we would like to sketch an outline of some previous research into this
problem. The aim of this outline is not to provide a historical survey but to situate our
result in the current theory of the obstacle problem.

The papers we discuss below (\cite{CPS04}, \cite{EP}, \cite{PS} \cite{EL12}) are technically very 
sophisticated and we have to refer the reader to the original sources for the full details. It should be mentioned that we will,
rather mischievously, slightly change the conceptual framework of the above papers into
the BMO framework of this paper in our explanations.

If $f\in L^\infty(Q_1^-)$ then it directly follows that $\tilde{D}^2u\in BMO(Q_{1/2}^-)$, i.e., 
\begin{equation}\label{BMOintro}
\|\tilde{D}^2 u -(\tilde{D}^2 u)_{Q_r^-(X^0)}\|_{L^2(Q_r^-(X^0))}\le C r^{(n+2)/2}\left(\|u\|_{L^1(Q_1^-)}+\|f\|_{L^\infty(Q_1^-)} \right), 
\end{equation}
where 
$(\tilde{D}^2 u)_{Q_r^-(X^0)}$ denotes the average over the parabolic half cylinder $Q_r^-(X^0)$ and $\tilde{D}^2$ is the second spatial and 
first time derivative (see the list of notation at the end of the introduction). If
$|(\tilde{D}^2 u)_{Q_r^-(X^0)}|\le C$ for some constant $C$ independent of $r$ and $X^0$ then by 
the triangle inequality $\|\tilde{D}^2 u \|_{L^2(Q_r^-(X^0))}\le C r^{(n+2)/2}$ which implies
that $\tilde{D}^2u\in L^\infty$. It is not difficult to see (cf. Lemma \ref{BMOLem}) that instead of 
subtracting $(\tilde{D}^2 u)_{Q_r^-(X^0)}$ in (\ref{BMOintro}) we can use 
$\tilde{D}^2 p_{u,r,X^0}(x,t)$ where $p_{u,r,X^0}(x,t)$ is a parabolic polynomial that is
second order homogeneous in $x$ and first order homogeneous in $t$. 
In particular it is enough to estimate $\|p_{u,r,X^0}\|_{L^\infty(Q_1^-)}$ in order to derive 
$W^{2,1}_\infty-$regularity for $u$.

The first regularity results for parabolic obstacle problems where obtained under the assumption that $u\ge 0$, which
implies that $p_{u,r,X^0}\ge -C$ in $Q_1^-$ since 
$\|u-p_{u,r,X^0}\|_{L^2(Q_r^-(X^0))}\le Cr^{2+(n+2)/2}$, by our BMO estimate. But if 
$p_{u,r,X^0}\ge -C$ and is a second order caloric polynomial then it directly follows that 
$p_{u,r,X^0}\le C$ in $Q_1^-$ which implies the optimal regularity by the above. The real difficulties
therefore occur for no-sign obstacle problems, i.e., when there is no assumption on the sign of the solution.

The first major breakthrough in the regularity theory for parabolic obstacle problems without a 
sign assumption was achieved in \cite{CPS04}, where it was proved that if \mbox{$f(x,t)=1$} then the solution 
to (\ref{main}) is in $W^{2,1}_\infty$. The proof is based on a monotonicity formula first proved
in \cite{caff93}. The monotonicity formula is applied on the positive and negative parts of
$D_e u$, the directional derivatives of $u$, which can be shown to be sub-caloric functions in their supports.
The monotonicity formula gives uniform bounds, in $r$, of the following averages
\begin{equation}\label{monquant}
\frac{1}{r^4}\int_{-r^2}^0\int_{B_r}|\nabla (D_e u)^+|^2G(x,-s)dxds \int_{-r^2}^0\int_{B_r}|\nabla (D_e u)^-|^2G(x,-s)dxds.
\end{equation}
Since, BMO-estimates implies
$u=p_{u,r,X^0}$ up to an error that is bounded in $W^{2,1}_\infty$ it follows that
the expression in (\ref{monquant}) is bounded with $p_{u,r,X^0}$ in place of $u$. Using that
$p_{u,r,X^0}$ is a parabolic second order polynomial implies, by elementary calculations, that
$p_{u,r,X^0}$ is bounded. 

Since the monotonicity formula in \cite{caff93} is valid only if $(D_e u)^\pm$ are sub-caloric 
functions the results in \cite{CPS04} are only valid in the case when $f$ is constant.
However, a refined version of the monotonicity formula, valid when 
$H(D_e u)^\pm \ge -C$ for some constant $C$,  was proved in \cite{EP}. This monotonicity formula
makes it possible to prove $W^{2,1}_\infty-$regularity for solutions when $f\in C^{0,1}$.
Clearly, if $f\notin C^{0,1}$ then $H(D_e u)^\pm\ge -C$ is no longer true which makes it 
difficult to use this method to prove regularity for $f$ less regular than Lipschitz.

The monotonicity formula approach therefore provides optimal regularity results for
$f\in C^{0,1}$. There is however a substantial and rather unsatisfying gap in the regularity theory.
If $f\in L^\infty$, then classical methods implies that $u\in W^{2,1}_p$ for any $p<\infty$. But in
order to achieve the $W^{2,1}_\infty-$regularity with the above mentioned methods, one needs to assume that $\nabla f\in L^\infty$, i.e., a whole extra derivative is required.

There is another approach to the regularity for the no-sign parabolic obstacle problem, which is based on an extra assumption on the behaviour of the free boundary. It is not difficult to
prove that the solution is $W^{2,1}_\infty$ close to points $X^0$ where 
$|\{u=0\}\cap Q_r^-(X^0)|>\epsilon |Q_r^-(X^0)|$ for every $r>0$. The most sophisticated 
result of this kind is \cite{PS} in the elliptic case and \cite{EL12} for the parabolic case.
The assumptions on the free boundary are, in order to be as week as possible, rather technical so 
we will have to refer the readers to the original papers for the details. For our purposes
it is enough to remark that even though the methods in \cite{PS} and \cite{EL12} are strong enough 
to 
prove optimal regularity of the solution - they are only able to do so under assumptions
on the free boundary which are unfortunately not verifiable in general.

The proof in our paper is based on the method in \cite{ALS11} where we prove similar results for the 
elliptic problem. Let us sketch the proof of Theorem \ref{main1} before we state the second
main theorem of the paper. If $u$ is a solution to (\ref{main}) then
we can write 
$$ u=v+g+\text{second order caloric polynomial},
$$
where $g$ is a solution to $H g= -f\chi_{\{u=0\}}$. We will slightly change the notation
and write the second order caloric polynomial as $S(u,r,X^0)p_{u,r,X^0}$ where
$\sup_{Q_1^-}|p_{u,r,X^0}|=1$, that is, $S(u,r,X^0)$ controls the norm of the caloric polynomial.
We will choose $S(u,r,X^0)p_{u,r,X^0}$ by means of a projection operator (see Definition 
\ref{defofPi}) that assures that $S(u,r,X^0)p_{u,r,X^0}$ closely approximates $u$. 
As before, the BMO estimates implies that $u\in W^{2,1}_\infty$ if and only if $S(u,r,X^0)$
is bounded.

The idea of the proof is to use that on $\Lambda=\{u=0\}$ we have $\tilde{D}^2 u=0$ and thus
\begin{equation}\label{idea}
0=\|\tilde{D}^2u\|_{L^2(\Lambda\cap Q_r^-(X^0))}\ge |S(u,r,X^0)|\|\tilde{D}^2p_{u,r,X^0}\|_{L^2(\Lambda\cap Q_r^-(X^0))}
\end{equation}
$$
-\|\tilde{D}^2 g\|_{L^2(\Lambda\cap Q_r^-(X^0))}-\|\tilde{D}^2v\|_{L^2(\Lambda\cap Q_r^-(X^0))}.
$$
Now, since $v\in W^{2,1}_\infty$ and since $\|p_{u,r,X^0}\|_{L^{\infty}(Q_1^-)}=1$, it follows that
$$
\|\tilde{D}^2p_{u,r,X^0}\|_{L^2(\Lambda\cap Q_r^-(X^0))}\le \sqrt{|\Lambda\cap Q_r^-(X^0)|}
$$ 
and
$$
\|\tilde{D}^2v\|_{L^2(\Lambda\cap Q_r^-(X^0))}\le \sqrt{|\Lambda\cap Q_r^-(X^0)|}\|\tilde{D}^2v\|_{L^\infty}.
$$
Moreover, $g$ may be written $g=\hat{g}+h$, where $h$ is a caloric function satisfying, by BMO 
estimates
$$\|\tilde{D}^2 h\|_{L^2(Q_r^-(X^0))}\le C r^{(n+2)/2}(\|u\|_{L^1(Q_1^-)}+\|f\|_{L^\infty(Q_1^-)}),$$
and $\hat{g}$ has zero boundary data and thus, by parabolic estimates 
$$\|\tilde{D}^2 \hat{g}\|_{L^2(Q_r^-(X_0)}\le C\|H\hat{g}\|_{L^2(Q_r^-(X^0))}\le C \sqrt{|\Lambda\cap Q_r^-(X^0)|}\|f\|_{L^\infty(Q_1^-)}.$$
If we disregard the caloric function $h$ then (\ref{idea}) can be written as
$$
|S(u,r,X^0)|\le C \|f\|_{L^\infty(Q_1^-)},
$$
which is our desired estimate. We may, unfortunately, not disregard $h$ and this explains 
why the paper is around 20 pages and not just a few lines.

In order to salvage something out of the above calculation we use the parabolic estimate
\begin{equation*}\label{desiredhest}
\|\tilde{D}^2 h\|_{L^\infty(Q_{r/2}^-(X_0))}\le Cr^{-(n+2)/2}\|\tilde{D}^2 h\|_{L^2(Q_r^-(X^0))},
\end{equation*} 
which means that we can use the calculation (\ref{idea}) in $Q_{r/2}^-(X^0)$. 
This estimate appears, then applied directly on $g=\hat{g}+h$, in (\ref{L1weakforgr}) and the 
calculations following (\ref{L1weakforgr}) explicates how this revised estimate controls
the dyadic decay of the measure $|\Lambda\cap Q_{2^{-j}}|$.
This is
the heart of the paper and carried out, with slight variations, in Proposition \ref{Mainprop}.
The dyadic decay of the measure $|\Lambda\cap Q_{2^{-j}}|$, Proposition \ref{Mainprop}, 
implies that if $S(u,r,X^0)$ is large enough then 
$$
\frac{|\Lambda\cap Q_{r/2}^-(X^0)|}{|Q_{r/2}^-(X^0)|}\le \frac{1}{4}\frac{|\Lambda\cap Q_{r}^-(X^0)|}{|Q_{r}^-(X^0)|},
$$
i.e., the function $\chi_\Lambda$ satisfies a Morrey space condition at $X^0$ - at least at the 
scale $r$. This is utilised, in Proposition \ref{DecayinLamimplreg}, to show that $g$
is indeed small as long as $S(u,r,X^0)$ is large in comparison to $\|u\|_{L^1(Q_1^-)}$
and $\|\tilde{D}^2 v\|_{L^\infty(Q_1^-)}$. For full details see the main body of the paper.

Our method improves on the previous research. It is strong enough to prove the optimal
regularity with minimal assumptions on $f$ and therefore improves on the monotonicity
formula approach in \cite{CPS04} or \cite{EP} that required $f\in C^{0,1}$. We do not make 
any assumptions on the solution or its free boundary (such as in \cite{PS} or \cite{EL12}). 
Furthermore,
whereas previous methods have utilised powerful, but rather specialized, monotonicity formulas
with limited reach, our method is based on standard $L^p$ estimates. This makes it 
likely that our methods can be extended to cover other equations, such as 
equations with variable coefficients, higher order equations
or fully non-linear equations.

As soon as $W^{2,1}_\infty-$regularity have been established we are in the position to
apply the powerful free boundary regularity results developed in \cite{CPS04}, \cite{EL12}
and \cite{LM12}. For that we need to assume that $f\in C^{\text{Dini}}$. This assumption is also optimal in the sense that if $f\not\in C^{\text{Dini}}$, then there is a time independent solution for which the free boundary is not $C^1$ at the origin, but in fact a spiral point, see \cite{Bla06}. We will only 
sketch the proof to highlight some minor differences, the reader is referred to the original papers
for the details.

If $u$ is a solution to (\ref{main}) then we define the free boundary, in the usual way, to be
\begin{equation}\label{FreeBoundaryDef}
\Gamma(u)=\partial_p\textrm{interior}\left(\{u=0\}\right),
\end{equation}
where $\dd_p$ stands for the parabolic boundary, referring to the points that are not parabolically interior, i.e. the points such that 
$$
Q_r^-\cap( \textrm{interior}\{u=0\})=\emptyset, 
$$
for any $r$ small enough. In this terminology our second main theorem is the following.

\begin{thm}\label{main2}
Let $u$ be a solution to (\ref{main}) and assume in addition that
$f\in C^{\textup{Dini}}(Q_1^-)$ and $f(0,0)=1$.
 \begin{itemize}
 \item If the origin is a low energy point 
(as in Definition \ref{reg}) then the free boundary $\Gamma(u)$ (as defined in 
(\ref{FreeBoundaryDef})) is, in a neighbourhood of the origin (which might depend on the solution $u$), a (parabolic) $C^1$ regular graph.
\item There is a modulus of continuity $\sigma$ and $r_0>0$ (both depending on $\|u\|_{L^1(Q_1^-)}$ and $\|f\|_{C^\text{Dini}(Q_1^-)})$) such that if
\begin{equation}\label{MDass}
\frac{\textup{MD}\left(\{x:u(x,-r^2)=0\}\cap B_r\right)}{r}>\sigma(r)
\end{equation}
for some $r<r_0$, then $\Gamma(u)$ is a $C^1$ regular graph in $Q_{r/2}^-$. Here $\textup{MD}$ stands for the minimal diameter.
\end{itemize}
\end{thm}

It is in order to explain the assumptions in Theorem \ref{main2} for the non-expert reader.
It is well known that the free boundary is not $C^1$ everywhere (see \cite{S} for an example in 
the time independent case) and some extra assumption
is needed to exclude that the origin is a singular point of the free boundary. A natural and correct assumption could be that the zero set of $u$ has positive Lebesgue density at the origin, which is a slightly stronger assumption that \eqref{MDass}.

There is also a different statement, based on a monotonicity formula, that
excludes singularities. The powerful monotonicity formula states that a certain energy
$W(r;u,f,(0,0))$ (defined in Section \ref{SecRegFB}) is almost increasing in $r$ if $u$ is a 
solution to (\ref{main}) and $(0,0)\in \Gamma(u)$. Furthermore, $W(0^+;u,f,(0,0))$ (which is well defined due to the almost monotonicity) can only assume the values 15, 15/2 or 0, and 
the value carries geometric information of the
free boundary at the origin. For reasons explained in \cite{CPS04}, the value $0$ in does 
not occur if $(0,0)\in \Gamma(u)$. If the value is $15/2$, we say that the origin is a low energy point.
In particular, by the discreteness of the limiting energies we can conclude that
if (\ref{MDass}) or a Lebesgue density condition holds then we are at a low energy point. The energy condition we impose is
therefore weaker than a Lebesgue density condition but comparable to \eqref{MDass}.


\subsection{Notation and assumptions} \label{sec:notation}
Throughout the paper we use the following notation:
$$
\begin{array}{ll}
X=(x,t), X^0=(x^0,t^0), Y=(y,s)& \text{- general points in $\R^n\times \R^+$}\\
u_t=\dd_t u =\frac{\partial u}{\partial t}&\textup{- the time derivative}\\
\nabla u=(\frac{\dd u}{\dd x_1},\ldots, \frac{\dd u}{\dd x_n})&\text{- the spatial gradient}\\
\lap u=\sum_{i=1}^n \frac{\dd^2 u}{\dd {x_i}^2} &\textup{- the Laplace operator}\\
Hu= \Delta u -u_t &\textup{- the heat operator}\\
Q_r(x_0,t_0)=B_r(x_0)\times (t_0-r^2,t_0+r^2) &\textup{- a parabolic cylinder}\\
Q_r^-(x_0,t_0)=B_r(x_0)\times (t_0-r^2,0]&\textup{- a half cylinder}\\
Q_r = Q_r(0,0), \quad Q_r^-=Q_r^-(0,0), \quad \partial_p Q_r^- = \partial_p Q_r^-(0,0)&\textup{- simplified notation}\\
W^{2,1}_p(A) & \text{- the Sobolev space of functions in }\\
& L^p(A) \text{ with }p\text{ integrable} \\
& \text{second derivatives in space and} \\
& \text{first derivatives in time.}\\
\tilde D^2 u =\left(\begin{array}{lr} D^2 u & 0\\
0 & \dd_t u 
\end{array}\right) \quad & \text{- the parabolic second derivative}\\
\big|\tilde D^2 u\big|^2=\sum_{i,j} \big| D^2_{i,j} u\big|^2+|\dd_t u|^2&\text{- the matrix norm used}\\
\Lambda = \{u=0\}\cap Q_1^-  & \text{- the coincidence set}\\
\Gamma = \dd_p(\{u\neq 0\}^\circ)\cap Q_1^-&\text{- the free boundary, the part of $\dd(\{u\neq 0\}^\circ)$}\\
 & \text{    that is not parabolically interior in $\Lambda$}\\
\lambda_r=\frac{|\Lambda\cap Q_r^-|}{|Q_r^-|} & \text{- the density of $\Lambda$ at the level $r$}\\
 \left(f\right)_{r,X^0}=\frac{1}{|Q_r^-(X^0)|}\int_{Q_r^-(X^0)} f \d x \d t &\text{- the mean value of $f$ over $Q_r^-(X_0)$}\\
 \left(f\right)_S=\frac{1}{|S|}\int_S f\d x\d t&\text{- the mean value of $f$ over $S$}
 \end{array}
$$
\begin{definition}[Dini continuity]
A function $f(x,t)$ is said to be Dini continuous (in the parabolic setting) if
$$
|f(x,t)-f(y,s)|\leq \sigma(\sqrt{|x-y|^2+|t-s|}), 
$$
where $\sigma$ is a non-negative continuous function such that $\sigma(0)=0$ and
$$
\int_0^1 \frac{\sigma(s)}{s}\d s<\infty.
$$
\end{definition}

\section{Parabolic spaces}
Here we present some useful definitions and result for parabolic spaces.
We will use the standard notation
$$
(f)_S=\frac{1}{|S|}\int_S f\d x\d t,
$$
for the average of a function $f$ over a set $S$.

\begin{definition}[Parabolic BMO]
We say that a function $f\in L^2(\Omega)$ is in $BMO(\Omega)$ if
$$
\|f\|_{BMO(\Omega)}^2\equiv \sup_{(x,t)\in \Omega,r>0}
\frac{1}{|Q_r(x,t)|}\int_{Q_r(x,t)\cap \Omega} |f(y,s)-(f)_{Q_r(x,t)}|^2+\|f\|_{L^2(\Omega)}^2<\infty.
$$
\end{definition}

The result below is well known and can be found in for instance \cite{Cam66}.
\begin{thm}[BMO-estimates for the heat equation]\label{CZtheory}
Let $H w= f$ in $Q_R^-$.
 If $f\in L^\infty(Q_R^-)$ then 
$$
\|D^2 w\|_{BMO(Q_{R/2}^-)}+\|\dd_t w\|_{BMO(Q_{R/2}^-)}\leq  C\big( \|f\|_{L^\infty(Q_R^-)}+\|w\|_{L^1(Q_R^-)}\big).
$$
Here the constant depends only on the space dimension.
\end{thm}

We will need the following standard results:

\begin{lem}\label{heateqest} Let $u$ be a solution of $Hu=0$ in $Q_r^-$. Then
$$
\|\tilde{D}^2 u\|_{L^\infty(Q_{r/2}^-)}\leq \frac{C}{r^2}\|u\|_{L^1(Q_r^-)}.
$$
\end{lem}
The proof of this lemma is contained in the proof of Theorem 8 on page 59 in \cite{Eva10}.

\begin{lem}\label{716inlieb}
Let 
$$
\begin{array}{ll}
Hu= f & \textrm{ in }Q_r^- \\
u=0 & \textrm{ on } \partial_p Q_r^-
\end{array}
$$
then
$$
\|\tilde{D}^2 u\|_{L^2(Q_r^-)}\le Cr^{(n+2)/2}\|f\|_{L^2(Q_r^-)}.
$$
\end{lem}
This is a rescaled version of Corollary 7.16 in \cite{Lieberman}.

\begin{lem}\label{RegHeatEQ}
Let 
$$
Hw=f \text{ in $Q_{1}^-$},
$$
where $f=Hv$ for some $v$ satisfying $\|\tilde{D}^2 v\|_{L^\infty(Q_{1})}<\infty$. Then there exists a constant $C$ such that 
$$
\|\tilde{D}^2 w\|_{L^\infty(Q_{1/2}^-)}\le C\left(\|w\|_{L^1(Q_{1}^-)}+\|\tilde{D}^2 v\|_{L^\infty(Q_{1}^-)} \right).
$$
\end{lem}

Below follows a parabolic version of the Poincar\'e inequality. 
\begin{lem}[Parabolic Poincar\'e inequality]\label{ParabolicPoincare}
Assume that $w\in L^1(Q_1^-(0))$. Then for some $\kappa \in (7/8,1)$ and some $C>0$
$$
\left\| w- (w)_{Q_\kappa^-}-x\cdot (\nabla w)_{Q^-_\kappa}\right\|_{L^2(Q_{\kappa}^-)}\le C\left(\|D^2 w\|_{L^2(Q_1^-)}+\| w_t\|_{L^2(Q^-_1)}\right).
$$
\end{lem}
\begin{proof} We notice that from the Poincar\'e inequality applied of each $t$-section of $w$
we may deduce that
\begin{equation}\label{D2PoinOnTimeSlice}
\int_{Q_\kappa^-}\left| w-(w)_{B_\kappa\times \{t\}}-x\cdot (\nabla w)_{B_\kappa \times \{t\}}\right|^2dxdt  \le C\int_{Q^-_\kappa}\left| D^2 w\right|^2dxdt.
\end{equation}
Hence, it suffices to find proper estimates for the differences
$$
\|(w)_{B_\kappa\times \{t\}}-(w)_{Q_k^-}\|_{L^2(Q_\kappa^-)},\quad \|x\cdot(\nabla w)_{B_\kappa\times \{t\}}-x\cdot (\nabla w)_{Q_\kappa^-}\|_{L^2(Q_\kappa^-)}.
$$
For that purpose we first of all remark that from the Poincar\'e inequality employed for fixed $x$ one can conclude
\begin{equation}\label{fixx}
\int_{Q_\kappa^-}|w-(w)_{(-\kappa^2,0)}|^2\d x\d t\leq C\int_{Q_\kappa^-}|w_t|^2\d x\d t.
\end{equation}
From H\"older's inequality we can then conclude
\begin{equation*}
\int_{Q_\kappa^-}|(w)_{B_\kappa\times \{t\}}-(w)_{Q_k^-}|^2\d x \d t\leq 
\end{equation*}
\begin{equation*}
\leq C\int_{Q_\kappa^-}\left(\int_{B_\kappa}|w-(w)_{(-\kappa^2,0)}|^2\right)\d x\d t\leq C\int_{Q_\kappa^-}|w_t|^2\d x \d t.
\end{equation*}

Now we compute, and use for the third equality that $x\cdot \nabla_y w(y,\cdot)=\textrm{div}_y(xw(y,\cdot))$ together with the divergence theorem
\begin{align*}
&\int_{Q_\kappa^-}\left| x\cdot (\nabla w)_{B_\kappa\times\{t\}}-x\cdot (\nabla w)_{Q_\kappa^-}\right|^2\d x\d t\\
&=\int_{Q_\kappa^-}\left| x\cdot \frac{1}{|B_\kappa|}\int_{B_\kappa}\nabla_y w(y,t)dy-x\cdot\frac{1}{\kappa^2|B_\kappa|}\int_{Q^-_\kappa}\nabla_y w(y,s)dyds\right|^2\d x\d t\\
&=\int_{Q_\kappa^-}\left| \frac{1}{|B_\kappa|}
\int_{\partial B_\kappa}x\cdot \nu w(y,t)dy-
\frac{1}{\kappa^2|B_\kappa|}\int_{-\kappa^2}^0\int_{\partial B_\kappa} 
x\cdot \nu w(y,s)dyds\right|^2\d x\d t\\
&\leq C
\int_{Q_\kappa^-}\int_{\partial B_\kappa}\left| w(y,t)-(w(y,s))_{s\in (-\kappa^2,0)}\right|^2 \d y \d x\d t\\
&\leq C\int_{-\kappa^2}^0 \int_{\partial B_\kappa}\left| w(y,t)-(w(y,s))_{s\in (-\kappa^2,0)}\right|^2\d t.
\end{align*}
Using polar coordinates and the mean value theorem, there exist a $\kappa \in (7/8,1)$ such that
\begin{align*}
&\frac{C}{|B_\kappa|}\int_{-\kappa^2}^0 \int_{\partial B_\kappa}\left| w(y,t)-(w(y,s))_{s\in (-\kappa^2,0)}\right|^2\d y\d t\\
&\le C
\int_{-1}^0 \int_{B_1}\left| w(y,t)-(w(y,s))_{s\in (-\kappa^2,0)}\right|^2\d y\d t\\
&\leq C\int_{Q_1^-}|w_t|^2\d x\d t,
\end{align*}
where the last inequality follows from \eqref{fixx}. Hence, 
\begin{equation}\label{gradfixest}
\int_{Q_k^-}|x\cdot(\nabla w)_{B_\kappa\times \{t\}}-x\cdot (\nabla w)_{Q_\kappa^-}|^2\d x\d t\leq C\int_{Q_1^-}|w_t|^2\d x\d t.
\end{equation}
To conclude the lemma, we only need to combine \eqref{D2PoinOnTimeSlice}, \eqref{fixx} and \eqref{gradfixest}.

\end{proof}

\section{The projection $\Pi$ and some technical results}
\begin{definition}\label{defofPi}
Let $\Pi(u,r,X^0)$ be the projection of $u$ into the space
of parabolic homogeneous caloric polynomials in $Q_r^-(X^0)$. In other words
$$
\int_{Q_r^-(X^0)}\big| \tilde D^2u(X)-\tilde D^2\Pi(u,r,X^0)\Big|^2=\inf_{p\in \mathcal{P}_2}
\int_{Q_r^-(X^0)}\big| \tilde D^2u(X)-\tilde D^2p(X)\big|^2,
$$
where $\mathcal{P}_2$ is the space of caloric polynomials that are homogeneous of degree two in space and homogeneous of degree one in time and where we use the matrix norm
$$
\big|A|^2=\sum_{i,j} \big| A_{i,j} \big|^2.
$$
Moreover, we will use the notation
$$
\Pi(u,r,X^0)=S(u,r,X^0)p_{u,r,X^0}(X),
$$
where $p_{u,r,X^0}(X)$ is a second order homogeneous caloric polynomial such that 

$$
\|\tilde D^2 p_{u,r,X^0}\|_{L^\infty(Q_1^-)}=\sup_{Q_1^-}|\tilde D^2 p_{u,r,X^0}|=1,
 $$
 and $S(u,r,X^0)\in \mathbb{R}_+$. 
\end{definition}
Below are certain properties of the projection that can be easily verified.
\begin{lem}\label{trivialPiprop} Let $u$ be as in Theorem \ref{main}. Then
\begin{enumerate}
\item $\Pi(\cdot,r,X^0)$ is linear,
\item $\|\Pi(u,r,X^0)\|_{L^2(Q_1^-)}\le C\|\D^2 u\|_{L^2(Q_1^-)}$ for $r\in [\frac12,1]$,
\end{enumerate}
In the above, $C$ is a constant depending only on the dimension.
\end{lem}
\begin{proof} The first statement follows from the fact that $\Pi$ is a projection.
In order to prove the second statement we observe that  if $\tilde D^2 u\in L^2(Q_1^-)$ then
$$
\inf_{P\in \mathcal{P}_2}\int_{Q_1^-}\big|\tilde D^2 u-\tilde D^2 P\big|^2\leq \int_{Q_1^-}\big|\tilde D^2 u|^2
$$
and moreover
$$
\|\tilde D^2 u-\tilde D^2 P\|_{L^2(Q_1^-)}\geq \|\tilde D^2 P\|_{L^2(Q_1^-)}-\|\tilde D^2 u\|_{L^2(Q_1^-)}.
$$
Thus, 
$$
\|\tilde D^2 \Pi(u,1,0)\|_{L^2(Q_1^-)}\leq 2\|\tilde D^2 u\|_{L^2(Q_1^-)}.
$$
Since $\tilde D^2 \Pi(u,1,0)$ is constant and $\Pi(u,1,0)$ homogeneous  we can conclude
$$
\| \Pi(u,1,0)\|_{L^2(Q_1^-)}\leq C\|\tilde D^2 u\|_{L^2(Q_1^-)}.
$$
The third statement now follows by a simple change of variables.

\end{proof}

\begin{lem}\label{BMOLem}
Assume $Hu\in L^\infty(Q_1^-)$. Then for every $X^0=(x^0,t^0)\in  Q_{1/2}^-$ and $r<\frac14$, the following inequality holds
$$
\Big\|\D^2\Big(\frac{u(rx+x^0,r^2t+t^0)}{r^2}-\Pi(u,r,X^0)\Big) \Big\|_{L^2(Q_1^ -)}\le C\big( \|u\|_{L^1(Q_1^-)}+\|H u\|_{L^\infty(Q_1^-)}\big),
$$
where $C$ depends only on the dimension.
\end{lem}

\begin{proof} Recall that
$$
\tilde D^2 u =\left(\begin{array}{lr} D^2 u & 0\\
0 & \dd_t u 
\end{array}\right).
$$From Theorem \ref{CZtheory} it follows that 
\begin{equation}\label{ivar}
\big\|\tilde D^2 u-(\tilde D^2 u)_{r,X^0}\|_{L^2(Q_r^-(X^0))}\le 
C\big( \|u\|_{L^1(Q_1^-)}+\|H u\|_{L^\infty}(Q_1^-)\big)r^{(n+2)/2}.
\end{equation}
We also observe that
\begin{align*}
\-\int_{Q_r^-(X^0)}\tilde D^2 u &= \-\int_{Q^ -_r(X^0)}\left(\left(\tilde D^2 u-\frac{H u}{n+1}I\right) +\frac{H u}{n+1}I\right)\\=M(X^ 0,r)+\-\int_{Q^-_r(X^0)}\frac{H u}{n+1}I, 
\end{align*}
where $M(X^0,r)$ is a constant matrix with zero trace and $I$ the identity matrix. Thus, if
$$
q_{X^0,r}=\frac12x^T [M(X^0,r)]_{n\times n} x+t[M(X^0,r)]_{n+1,n+1},
$$
then
$$
\tilde D^2q_{X^0,r}=M(X^0,r).
$$
It follows that
$$
(\tilde D^2 u)_{r,X^0}=\-\int_{Q_r^-(X^0)}\tilde D^2 u=\tilde D^2 q_{X^0,r}+\-\int_{Q_r^-(X^0)}\frac{H u}{n+1}I.
$$
Hence, 
$$
\Big\|\tilde D^2u-\tilde D^2q_{X^0,r}\Big\|_{L^2(Q_r^-(X^0))}\leq 
$$
$$
\leq\Big\|\tilde D^2 u-\-\int_{Q_r^-(X^0)}\tilde D^2 u \Big\|_{L^2(Q_r^-(X^0))}+\Big\|\frac{H u}{n+1}I\Big\|_{L^2(Q_r^-(X^0))}.
$$
From this, the definition of $\Pi$ and rescaling the inequality (\ref{ivar}), the 
lemma follows.\end{proof}

\section{The key proposition}
In this section we present Proposition \ref{Mainprop}, which is the base of the whole paper. First a lemma.
\begin{lem}\label{wrest}
Assume that $w$ solves
$$
\left\{\begin{array}{ll}
H w(x,t)=f(rx,r^2t) & \textrm{ in } Q^-_{1}, \\
w(x,t)=u_r(x,t)-\Pi(u,r,0) & 
\textrm{ on }  \dd_p Q^-_{1},
\end{array}\right.
$$
where
$$
u_r(x,t)=\frac{u(rx,r^2t}{r^2}.
$$
Then 
$$
\|\tilde D^2 w\|_{L^\infty(Q_{1/2}^-)}\le C\big( \|u\|_{L^1(Q^-_1)}+\|\tilde D^2 v\|_{L^\infty(Q^-_1)}\big).
$$
\end{lem}
\begin{proof} To simplify the notation write $S_r=S(u,r,0)$ and $p_r=p_{u,r,0}$. Define the function
$$
\tilde u = u_r-S_r p_r - (u_r-S_rp_r)_{Q_1^-}-x\cdot \left(\nabla (u_r-S_rp_r)_{Q_1^-}\right).
$$
Since Lemma \ref{BMOLem} implies
$$
\|\D^2 (u_r-S_rp_r)\|_{L^2(Q_1^-)}\leq C(\|u\|_{L^1(Q_1^-)}+\|H u\|_{L^\infty(Q_1^-)})
$$
we can use Lemma \ref{ParabolicPoincare} to obtain
\begin{equation}\label{utildeest}
\|\tilde u\|_{L^2(Q_\frac78^-)}\leq C(\|u\|_{L^1(Q_1^-)}+\|H u\|_{L^\infty(Q_1^-)}).
\end{equation}
Now we observe that $H \tilde u=H u$ so that interior estimates (Lemma \ref{RegHeatEQ}) combined with \eqref{utildeest} imply
$$
|(u_r-S_rp_r)_{Q_1^-}|+|x\cdot \left(\nabla (u_r-S_rp_r)\right)_{Q_1^-}|=|\tilde u_r(0)|+|\nabla \tilde u_r(0)|\leq
$$
$$
\leq C(\|u\|_{L^1(Q_1^-)}+\|H u\|_{L^\infty(Q_1^-)}).
$$
As a consequence
$$
\| u-S_rp_r\|_{L^2(Q_1^-)}\leq C(\|u\|_{L^1(Q_1^-)}+\|H u\|_{L^\infty(Q_1^-)}).
$$
To obtain the desired estimate, we use interior estimates (Lemma \ref{RegHeatEQ}) on $w$ together with the definition of $v$.
\end{proof}
We recall the notation
$$
\lambda_r=\frac{|\Lambda \cap Q_r^-|}{|Q_r^-|}.
$$

\begin{prop}\label{Mainprop}
Let $u$ be a solution to (\ref{main}). Then there exist $C_0$ and $C_1$
depending only on the dimension 
such that if $X^0\in \Lambda\cap Q^-_{1/2}$ and $r<\frac14$ then
$$
\frac{C_0\|\tilde D^2 v\|_{L^\infty(Q_1^-)}}{S(r,u,X^0)-C_1\big( \|u\|_{L^1(Q_1^-)}+\|\tilde D^2 v\|_{L^\infty(Q^-_1)}\big)}\lambda_r^{1/2}\ge
\lambda_{\frac{r}{2}}^{1/2},
$$
whenever
$$
S(r,u,X^0)> 2C_1\big( \|u\|_{L^1(B_1)}+\|\tilde D^2 v\|_{L^\infty(B_1)}\big).
$$
\end{prop}
\begin{proof} For simplicity let $X^0=0$ and 
$$
u_r(X)=w_r(X)+S(u,r,0)p_{u,r,0}(X)+g_r(X),
$$ 
where 
$$
\left\{
\begin{array}{ll}
H g_r(x,t)=-f(rx,r^2t)\chi_{\Lambda(u(rx,r^2t))} & \textrm{ in } Q_{1}^-, \\
g_r=0 & \textrm{ on }  \dd_p Q^-_{1},
\end{array}\right.
$$
and 
$$
\left\{\begin{array}{ll}
H w_r(x,t)=f(rx,r^2t) & \textrm{ in } Q^-_{1}, \\
w_r(x,t)=u_r(x,t)-S(u,r,0)p_{u,r,0}(X) & 
\textrm{ on }  \dd_p Q^-_{1}. 
\end{array}\right.
$$
From Lemma \ref{716inlieb} and Lemma \ref{wrest} it follows that 
\begin{equation}\label{L1weakforgr}
\| \tilde D^2 g_r\|_{L^2(Q^-_{1/2})}\le 
C\|f\|_{L^\infty}\|\chi_{\Lambda(u(rx,r^2t))}\|_{L^2(Q^-_{1})},
\end{equation}
and
\begin{equation}\label{Linftyforwr}
\|\tilde D^2 w_r\|_{L^\infty(Q_{1/2}^-)}\le C\big( \|u\|_{L^1(Q^-_1)}+\|\tilde D^2 v\|_{L^\infty(Q^-_1)}\big).
\end{equation}

At this stage we use that $D^2 u = \dd_t u = 0$ a.e. in $\{u=0\}$ and thus
$$
0=\|\D^2 u_{r} \|_{L^2(\Lambda_r\cap Q_\frac12^-)}=\|\D^2 (w_r+S(u,r,0)p_{u,r,0}+g_r) \|_{L^2(\Lambda_r\cap Q_\frac12^-)},
$$
which implies
\begin{equation}\label{firstest}
\|\D^2 S(u,r,0)p_{u,r,0} \|_{L^2(\Lambda_r\cap Q_\frac12^-)}\leq \|\D^2 w_r \|_{L^2(\Lambda_r\cap Q_\frac12^-)}+\|\D^2 g_r \|_{L^2(\Lambda_r \cap Q_\frac12^-)}.
\end{equation}
From the definition of $p_{u,r,0}$ it follows that
\begin{equation}\label{pL2est}
\|\D^2 p_{u,r,0} \|_{L^2(\Lambda_r\cap Q_\frac12^-)}\geq \|\D^2 p_{u,r,0} \|_{L^\infty(\Lambda_r\cap Q_\frac12^-)}\lambda_r^\frac12\geq c\lambda_{\frac{r}{2}}^\frac12.
\end{equation}
In addition, the estimates in \eqref{L1weakforgr} and \eqref{Linftyforwr} imply
\begin{equation}\label{grL2est}
\|\D^2 g_r \|_{L^2(\Lambda_r\cap Q_\frac12^-)}\leq C\|f\|_{L^\infty(\Lambda_r\cap Q_1^-)}\lambda_r^\frac12\leq C\|\D^2 v \|_{L^\infty(Q_1^-)}\lambda_r^\frac12.
\end{equation}
and
\begin{equation}\label{wrL2est}
\|\D^2 w_r \|_{L^2(\Lambda_r\cap Q_\frac12^-)}\leq C\big( \|u\|_{L^1(Q^-_1)}+\|\tilde D^2 v\|_{L^\infty(Q^-_1)}\big)\lambda_\frac{r}{2}^\frac12.
\end{equation}
Combining \eqref{pL2est}, \eqref{grL2est} and \eqref{wrL2est} we can conclude
$$
S(u,r,0)\lambda_{\frac{r}{2}}^\frac12\leq C\|\D^2 v \|_{L^\infty(Q_1^-)}\lambda_r^\frac12+C\big( \|u\|_{L^1(Q^-_1)}+\|\tilde D^2 v\|_{L^\infty(Q^-_1)}\big)\lambda_\frac{r}{2}^\frac12,
$$
from which the desired result follows.
\end{proof}

\begin{prop}\label{DecayinLamimplreg}
Let 
$$
\begin{array}{ll}
Hg = f(x,t)\chi_\Lambda & \textrm{ in } Q^-_{2^{-k}} \\
\|\tilde{D}^2 g\|_{L^2(Q_{2^{-k}})}\le C_12^{-k(n+2)/2} &
\end{array}
$$
where $\|f\|_{L^\infty}\le C_2$ and
\begin{equation}\label{geometriclambda}
\lambda_{2^{-j-1}}\le \frac{1}{4}\lambda_{2^{-j}}
\end{equation}
for $j=k,k+1,k+2,..., J$. Then there exists a universal constant $C_0$ such that
\begin{equation}\label{ghikop}
\sup_{Q_1}|\Pi(g, 2^{-j},0)| \le C_0\left( C_1+C_2\right)
\end{equation}
for $j=k,k+1,k+2,...,J$.
\end{prop}
\begin{proof} There is no loss of generality to assume that $k=0$.
In particular, we may parabolically rescale $g$ to $\bar{g}(x,t)=g(2^{-k}x,2^{-2k}t)/2^{-2k}$.
Then $\bar{g}$ satisfies the assumptions with $k=0$ and $J-k$ in place of $J$. If we can prove 
(\ref{ghikop}) for $\bar{g}$ then it follows for $g$ by scaling back. We may thus
assume that $k=0$ in the proof.

We may write $g_j=g|_{Q_{2^{-j}}^-}$, the restriction of $g$ to $Q_{2^{-j}}^-$, 
as the following sum
\begin{equation}\label{somegjdecomp}
g_j=\sum_{k=0}^j h_k + \tilde{g}_j
\end{equation}
where
$$
\begin{array}{ll}
H \tilde g_j =f\chi_{\Lambda} & \textrm{ in }Q_{2^{-j}}^- \\
\tilde g_j= 0 & \textrm{ on } \partial_p Q_{2^{-j}}^-
\end{array}
$$
and
$$
\begin{array}{ll}
H h_k =0 & \textrm{ in }Q_{2^{-k}}^- \\
h_k=\tilde  g_{k-1} & \textrm{ on } \partial_p Q_{2^{-k}}^-,
\end{array}
$$
where we, for consistency, identify $g_{-1}=g$. That $\tilde g_j=g$ in $Q_{2^{-j}}^-$
follows by an easy induction. It is true, by definition, for $j=-1$. If 
$\tilde g_{j-1}=g$ in $Q_{2^{-j+1}}^-$ then $H\tilde  g_j=H g$ in $Q_{2^{-j}}^-$ and 
$\tilde g_j=g$ on $\partial_p Q_{2^{-j}}^-$ by construction. By the maximum principle
it follows that $\tilde g_j=g$ in $Q_{2^{-j}}^-$ and our induction is complete.

Next we notice that by (\ref{geometriclambda}) and H\"older's inequality it follows that
$$
\|H \tilde g_j\|_{L^2(Q_{2^{-j}}^-)}\le C_22^{-j}|Q_{2^{-j}}^-|^{1/2}.
$$
This implies, by Lemma \ref{716inlieb}, that
\begin{equation}\label{W212forgj}
\|\tilde   g_j\|_{W^{2,1}_2(Q_{2^{-j}}^-)}\le C C_22^{-j}|Q_{2^{-j}}^-|^{1/2}
\end{equation}
for some universal constant $C$.

Since $h_{j+1}$ is caloric with $\tilde g_j$ as boundary values it follows from Lemma \ref{716inlieb} that
\begin{equation}\label{estonPihj}
\sup_{Q_1^-}|\Pi(h_{j+1},2^{-k},0)|\le C C_22^{-j}
\end{equation}
for some universal constant $C$ and all $k\ge j+1$ and all $j\ge 0$.

From (\ref{W212forgj}) we may deduce, using Lemma \ref{trivialPiprop} and a simple rescaling, that
\begin{equation}\label{estPigj}
\sup_{Q_1^-}|\Pi(\tilde g_{j},2^{-j},0)|\le C C_22^{-j}.
\end{equation}

We need to estimate the projection of $h_0$ as well. From parabolic estimates (cf. Lemma \ref{716inlieb}) we can conclude that 
$$
\|\tilde{D}^2 h_0\|_{L^2(Q_1^-)}\le C\|\tilde{D}^2 g\|_{L^2(Q_1^- )}
$$
and thus, using Lemma \ref{trivialPiprop} again, that
\begin{equation}\label{estPih0}
\sup_{Q_1^-}|\Pi(h_{0},2^{-k},0)|\le C C_1
\end{equation}
for any $k\ge 0$.

Using the linearity of $\Pi$ we can finally conclude that for any $j\le J$ 
$$
\sup_{Q_1^-}|\Pi(g,2^{-j},0)|\le \sum_{k=0}^j\sup_{Q_1^-}|\Pi(h_k,2^{-j},0)|+|\Pi(\tilde g_j,2^{-j},0)|\le
$$
$$
\le C\left( C_1+C_2\sum_{k=0}^j 2^{-k} +C_2\right)\le C\left( C_1+3C_2\right),
$$
where we have used (\ref{somegjdecomp}) in the first inequality and (\ref{estonPihj}), (\ref{estPigj}) and (\ref{estPih0}) in the second.
\end{proof}

\section{Proof of the main Result}\label{secproof}
The two following lemmata provides us with the result that if $S$ is bounded then $u$ is $W^{2,1}_\infty$.
\begin{lem}[Quadratic growth implies $W^{2,1}_\infty$]\label{quadc11}
Suppose $u$ is a solution of \eqref{main} such that 
$$
\sup_{Q_r^-(Y^0)}|u|\leq Mr^2
$$
for all $0<r<1/2$ and $Y^0\subset Q_\frac12^-\cap \dd\{u\neq 0\}$. Then
$$
\|\D^2 u\|_{L^\infty(Q_\frac12^-)}\leq C\left(M+\|u\|_{L^1(Q_1^-)}+\|D^2 v\|_{L^\infty(Q_1^-)}\right).
$$
\end{lem}
\begin{proof} Take $X^0\in Q_\frac12^-$ and define 
$$
r=\sup\{r:Q_r^-(X^0)\cap \dd\{u\neq 0\}=\emptyset\}.
$$
We now split the proof into two cases:

\noindent {\bf Case 1: $r<\frac14$}. From the definition of $r$ it follows that with 
$$
w=u-v-v(X^0)-\nabla v(X_0)\cdot (x^0-x),
$$
then $Hw=0$ in $Q_r^-(X^0)$. From Lemma \ref{heateqest}
$$
\|\tilde{D}^2 w\|_{L^\infty(Q_{r/2}^-(X^0))}\leq \frac{C}{r^2}\|u\|_{L^\infty(Q_r^-(X^0))}\leq C\left(M+\|D^2 v\|_{L^\infty(Q_r^-(X^0))}\right),
$$
and thus
$$
\|\tilde{D}^2 u\|_{L^\infty(Q_{r/2}^-(X^0))}\leq C\left(M+\|D^2 v\|_{L^\infty(Q_1^-)}\right),
$$
where $C$ is possibly a larger constant.

\noindent {\bf Case 2: $r\geq \frac14$}. In this case, it is clear that 
$$
w=u-v-v(X^0)-\nabla v(X_0)\cdot (x^0-x),
$$
then $Hw=0$ in $Q_\frac14^-(X^0)$. Lemma \ref{heateqest} then implies
$$
\|\tilde{D}^2 w\|_{L^\infty(Q_{1/8}^-(X^0))}\leq C\|u\|_{L^1(Q_\frac14^-(X^0))}\leq C\left(\|u\|_{L^1(Q_1^-)}+\|D^2 v\|_{L^\infty(Q_\frac14^-(X^0))}\right),
$$
which implies
$$
\|\tilde{D}^2 u\|_{L^\infty(Q_{1/8}^-(X^0))}\leq C\left(\|u\|_{L^1(Q_1^-)}+\|D^2 v\|_{L^\infty(Q_1^-)}\right).
$$
The combination of the two cases above yields the desired estimate.
\end{proof}
\begin{lem}[Bounded $S$ implies quadratic growth]\label{bddSquad}
Suppose $u$ is a solution of \eqref{main} and let $X^0\in  Q_\frac12^-\cap \partial \{u\ne 0\}$.
 Then for $r<1/4$
$$
\sup_{Q_{\frac{r}{2}}^-(X^0)} |u|\leq C\left(S(u,r,X^0)+\|u\|_{L^1(Q_1^-)}+\|H u\|_{L^\infty(Q_1^-)}\right)r^2.
$$
\end{lem} The proof is very similar to the proof of Lemma \ref{wrest} and therefore we give only a sketch of the proof.
\begin{proof} To avoid cumbersome writing we use the notation $S=S(u,r,X^0)$ and 
$$u_r=u_{r,X^0}=\frac{u(rx+x^0,r^2t+t^0)}{r^2},
$$
where $X^0=(x^0,t^0)$, throughout the whole proof. The hypotheses of the lemma and Lemma \ref{BMOLem} imply
$$
\|\D^2 u_{r}\|_{L^2(Q_1^-)}\leq C\left(S+\|u\|_{L^1(Q_1^-)}+\|H u\|_{L^\infty(Q_1^-)}\right).
$$
Defining
$$
\tilde u = u_r - (u_r)_{Q_1^-}-x\cdot \left(\nabla (u_r)_{Q_1^-}\right),
$$
we can reason as in the proof of Lemma \ref{wrest} to obtain that 
\begin{align*}
\|\tilde u\|_{L^\infty(Q_\frac12^-)}+|(u_r)_{Q_2^-}|+|x\cdot \left(\nabla (u_r)\right)_{Q_2^-}|&=\|\tilde u\|_{L^\infty(Q_\frac12^-)}+|\tilde u_r(0)|+|\nabla \tilde u_r(0)|\\
&\leq C\left(S+\|u\|_{L^1(Q_1^-)}+\|H u\|_{L^\infty(Q_1^-)}\right).
\end{align*}
This implies, by the triangle inequality
$$
\|u_r\|_{L^\infty(Q_\frac12^-)}\leq C(S+\|u\|_{L^1(Q_1^-)}+\|H u\|_{L^\infty(Q_1^-)}).
$$
\end{proof}

We are now ready to give the proof of the main theorem.

\begin{proof}[~Proof of Theorem \ref{main1}]

In view of Lemma \ref{bddSquad} and Lemma \ref{quadc11} it is enough to prove that 
\begin{equation}\label{aiminmain}
\sup_{Q_1^-}|\Pi(u,2^{-j},X^0)|\leq C(\|u\|_{L^1(Q_1^-)}+\|D^2v\|_{L^\infty(Q_1^-)}),
\end{equation}
for $j=2,3,4,...$ and some universal constant $C$ and every $X^0\in Q_{1/2}^-$. 
It is enough to prove (\ref{aiminmain}) for $X^0=0$. Once (\ref{aiminmain}) is proved for $X^0=0$
a translation argument assures that (\ref{aiminmain}) holds for any $X^0\in Q_{1/2}^-$. Then 
Lemma \ref{bddSquad} provides a quadratic bound
on the solution which, by Lemma \ref{quadc11} implies regularity.

Let us denote by $B$ the set $B\subset \mathbb{N}$ of all $j\in \mathbb{N}$ such that 
\begin{equation*}
S(u,2^{-j},0)\ge 2C_0\|\tilde{D}^2 v\|_{L^\infty(Q_1^-)}+2C_1\left(\|u\|_{L^1(Q_1^-)}+\|\tilde D^2v\|_{L^\infty(Q_1^-)}  \right),
\end{equation*}
where $C_0$ and $C_1$ are as is Proposition \ref{Mainprop}. Naturally if $j\notin B$
then 
$$
S(u,2^{-j},0)\le 2C_0\|\tilde{D}^2 v\|_{L^\infty(Q_1^-)}+2C_1\left(\|u\|_{L^1(Q_1^-)}+\|\tilde D^2v\|_{L^\infty(Q_1^-)}  \right)
$$
for all $j$ which in turn implies (\ref{aiminmain}) with $C=2(c_0+C_1)$. 
Therefore $B$ consists of the {\sl ``bad''}
scales where (\ref{aiminmain}) might not hold.

If $j\in B$ then, according to Proposition \ref{Mainprop}, 
\begin{equation}\label{2star10}
\frac{1}{4}\lambda_{2^{-j}}\ge \lambda_{2^{-j-1}}.
\end{equation}

Moreover, if $\{k,k+1,k+2,...,J\}\subset B$ then (\ref{2star10}) holds for all
$j\in \{k,k+1,...,J\}$. In particular, if $\{k,k+1,k+2,...,J\}\subset B$ and if we split
$u$ into
\begin{equation}\label{splitinmain}
u=\tilde{v}+g
\end{equation}
where $H\tilde{v}=Hv$ and $Hg=-f\chi_\Lambda$ then $g$ satisfies the conditions in Proposition 
\ref{DecayinLamimplreg} for $j\in \{k,k+1,...,J\}$. We have a choice in $h$ and $g$ and we 
may choose $g$ such that $\| g\|_{L^1(Q_1^-)}\le C\|f\|_{L^\infty(Q_1^-)}$ by for instance 
letting $g$ be the convolution of the heat kernel and $-f\chi_{\Lambda}$. Then Theorem \ref{CZtheory}
implies that
\begin{equation}\label{BMforgMain}
\|\tilde{D} g\|_{BMO(Q_{1/2}^-)}\le C\|f\|_{L^\infty(Q_1^-)}.
\end{equation}
Then it follows, by the triangle inequality, that 
\begin{equation}\label{callthissomething}
\|\tilde{v}\|_{L^1(Q_1^-)}\le \|u\|_{L^1(Q_1^-)}+C\|f\|_{L^\infty(Q_1^-)}.
\end{equation}

Let $j\in \mathbb{N}$. Then either $j\notin B$ and (\ref{aiminmain}) holds or
$j\in B$ and there exists a smallest $k\in \mathbb{N}$ such that $\{k,k+1,k+2,...,j\}\subset B$.
Using (\ref{splitinmain}), the linearity of $\Pi$ and the triangle inequality we can estimate
\begin{equation}\label{triangleinmain}
\sup_{Q_1^-}\left| \Pi(u,2^{-j},0)-\Pi(u,2^{-k},0)\right|\le
\sup_{Q_1^-}\left| \Pi(\tilde{v},2^{-j},0)-\Pi(\tilde{v},2^{-k},0)\right|+
\end{equation}
$$
+\sup_{Q_1^-}\left| \Pi(g,2^{-j},0)-\Pi(g,2^{-k},0)\right|.
$$
Since $v\in W^{2,1}_\infty(Q_1^-)$ it follows from Lemma \ref{heateqest} that
\begin{align}\label{tildevprojest}
\sup_{Q_1^-}\left| \Pi(\tilde{v},2^{-j},0)\right|,\sup_{Q_1^-}\left|\Pi(\tilde{v},2^{-k},0)\right|&
\le C\left( \|\tilde{D}^2v\|_{L^\infty(Q_1^-)}+\|\tilde{v}\|_{L^1(Q_1^-)}\right)\\
\nonumber &\le 
C\left( \|\tilde{D}^2v\|_{L^\infty(Q_1^-)}+\|u\|_{L^1(Q_1^-)}\right),
\end{align}
where we used (\ref{callthissomething}) in the last inequality.

We use Proposition \ref{DecayinLamimplreg} to estimate
\begin{equation}\label{gprojest1}
\sup_{Q_1^-}\left| \Pi(g,2^{-j},0)-\Pi(g,2^{-k},0)\right|
=\sup_{Q_1^-}\left| \Pi(\tilde{g},2^{-j},0)\right|
\end{equation}
where $\tilde{g}=g-\Pi(g,2^{-k},0)$. In particular, by (\ref{BMforgMain}) and Lemma \ref{BMOLem}, it follows that
$$
\|\tilde{g}\|_{L^2(Q_{2^{-k}}^-)}\le C\|f\|_{L^\infty(Q_1^-)}2^{-k(n+2)/2}. 
$$
Moreover, by (\ref{2star10}) and our assumption that $\{k,k+1,...,j\}\subset B$ the assumptions in 
Proposition \ref{DecayinLamimplreg} are satisfied with $C_1=C\|f\|_{L^\infty(Q_1^-)}$
for some universal constant $C$. From Proposition \ref{DecayinLamimplreg} and
(\ref{gprojest1}) it therefore follows that
\begin{equation}\label{gprojest2}
\sup_{Q_1^-}\left| \Pi(g,2^{-j},0)-\Pi(g,2^{-k},0)\right|
=\sup_{Q_1^-}\left|\Pi(\tilde{g},2^{-j},0)\right|\le C \|f\|_{L^\infty(Q_1^-)}.
\end{equation}
From (\ref{triangleinmain}), (\ref{tildevprojest}) and (\ref{gprojest2})
we can conclude that
\begin{equation}\label{almostdone}
\sup_{Q_1^-}\left| \Pi(u,2^{-j},0)\right|\le \sup_{Q_1^-}\left|\Pi(u,2^{-k},0)\right|+
C\left( \|\tilde{D}^2 v\|_{L^\infty(Q_1^-)}+\|u\|_{L^1(Q_1^-)}\right).
\end{equation}

We need to estimate $\sup_{Q_1^-}\left|\Pi(u,2^{-k},0)\right|$. Remember that
$k\in \mathbb{N}$ was the smallest constant such that $\{k,k+1,...,j\}\subset B$. This
implies that $k-1 \notin B$ so 
\begin{equation}\label{Sngt}
S(u,2^{-k+1},0)\le 2C_0\|\tilde{D}^2 v\|_{L^\infty(Q_1^-)}+2C_1\left(\|u\|_{L^1(Q_1^-)}+\|\tilde D^2v\|_{L^\infty(Q_1^-)}  \right).
\end{equation}
Furthermore, by Lemma \ref{BMOLem},
\begin{align*}
\|\tilde{D}^2 (u-\Pi(u,2^{-k+1},0))\|_{L^2(Q_{2^{-k+1}}^-)}\\ \le C2^{-(n+2)(k-1)/2}\left(\|f\|_{L^\infty(Q_1^-)}+\|u\|_{L^1(Q_1^-)} \right).
\end{align*}
That is
\begin{align}
\nonumber \sup_{Q_1^-}\left|\Pi\left(u-\Pi(u,2^{-k+1},0),2^{-k},0 \right) \right| \le \\
\nonumber C \|\Pi\left(u-\Pi(u,2^{-k+1},0),2^{-k},0 \right) \|_{L^2(Q_1^-)} \le \\
\label{oddoneout} 2^{-k-k(n+2)/2}C\|\Pi\left(u-\Pi(u,2^{-k+1},0),2^{-k},0 \right) \|_{L^2(Q_{2^{-k}}^-)} \le \\
\nonumber 
\le 2^{-(n+2)k/2}C\|\tilde{D}^2 (u-\Pi(u,2^{-k+1},0))\|_{L^2(Q_{2^{-k}}^-)}\\
\nonumber 
\le C\left(\|f\|_{L^\infty(Q_1^-)}+\|u\|_{L^1(Q_1^-)} \right),
\end{align}
where we used that $\Pi(\cdot)$ is a polynomial in the first inequality, a rescaling in the 
second, Lemma \ref{trivialPiprop} in the third and Lemma \ref{BMOLem} in the last inequality.
The triangle inequality, \eqref{Sngt} and (\ref{oddoneout}) implies that
$$
\sup_{Q_1^-}\left|\Pi(u,2^{-k},0)\right|\le
$$
\begin{equation}\label{estu2-kscale}
\le C2^{(n+2)/2}\left(\|f\|_{L^\infty(Q_1^-)}+\|u\|_{L^1(Q_1^-)} \right)+ \sup_{Q_1^-}\left|\Pi(u,2^{-k+1},0)\right|\le
\end{equation}
$$
\le C \left(\|f\|_{L^\infty(Q_1^-)}+\|u\|_{L^1(Q_1^-)} \right).
$$
Using (\ref{estu2-kscale}) in (\ref{almostdone}) yields
\begin{equation}\label{theestimatewewant1}
\sup_{Q_1^-}\left| \Pi(u,2^{-j},0)\right|\le 
C\left( \|\tilde{D}^2 v\|_{L^\infty(Q_1^-)}+\|u\|_{L^1(Q_1^-)}\right),
\end{equation}
where we also used that $\|f\|_{L^\infty}\le C\|\tilde{D}^2 v\|_{L^\infty}$.
Equation (\ref{theestimatewewant1}) is valid for all $j\in \mathbb{N}$.

By translating the coordinate system it follows from (\ref{theestimatewewant1}) that
\begin{equation}\label{theestimatewewant2}
\sup_{Q_1^-}\left| \Pi(u,2^{-j},X^0)\right|\le 
C\left( \|\tilde{D}^2 v\|_{L^\infty(Q_1^-)}+\|u\|_{L^1(Q_1^-)}\right),
\end{equation}
for any $X^0\in Q_{1/2}^-$.

Lemma \ref{bddSquad} and (\ref{theestimatewewant2}) implies, for all $r\in (0,1/4)$, that
$$
\sup_{Q_r^-(X^0)}|u|\le C\left(\|\tilde{D}^2 v\|_{L^\infty(Q_1^-)}+\|u\|_{L^1(Q_1^-)} \right).
$$
which by Lemma \ref{quadc11} implies that
$$
\|\tilde{D}^2u\|_{L^\infty(Q_{1/2}^-)}\le C\left(\|\tilde{D}^2 v\|_{L^\infty(Q_1^-)}+\|u\|_{L^1(Q_1^-)} \right),
$$
and the proof is complete.
\end{proof}

\section{Regularity of the free boundary}\label{SecRegFB}
In this section we prove the second main theorem. The idea is to prove that at low energy points (see Definition \ref{reg} below) points and at a scale small enough, the solution is non-negative.
\subsection{Weiss' monotonicity formula}

In order to prove the regularity of the free boundary we need to introduce some notions from for instance \cite{CPS04}.

Define the Weiss energy for $v(x,t)$: $\R^n \times \R^- \rightarrow \R$ to be
$$
W(r;v,f,(x^0,t^0))=\frac{1}{r^4}\int_{\R^n\times (-r^2+t^0,t^0]} \left(|\nabla v|^2+2fv+\frac{v^2}{t}\right)G(x,-t) dx d t,
$$
and let
$$
v_{r,(x^0,t^0)}(x,t)=\frac{v(rx+x^0,r^2t+t^0)}{r^2},
$$
so that 
$$W(r;v,f,(x^0,t^0))=W(1;v_{r,(x^0,t^0)},f_{r,(x^0,t^0)},0),$$
 where 
 $$f_{r,(x^0,t^0)}(x,t)=f(rx+x^0,r^2t+t^0).$$
 Moreover, for a general function $u(x,t)$ define 
\begin{equation}Lu=x\cdot \nabla u+2tu_t-2u.
\end{equation}

The following proposition is a parabolic version of Weiss' monotonicity formula. Once we know that $u$ enjoys the optimal regularity (Theorem \ref{main1}), it can be proved in the same manner as in \cite{EL12} or \cite{LM12}.
\begin{prop}\label{prop:mono} 
Let $u$ be a solution of \eqref{main}, $\psi\in C_0^\infty(B_{3/4})$ such that $\psi=1$ on $B_{1/2}$, and set $v=u\psi$. Then there is a continuous function 
$$F=F(\|f\|_{C^\text{Dini}}, \|u\|_{L^1(Q_1^-)},r),
$$
with $F(0)=0$ such that
$$
W(r;v,f,(x^0,t^0))+ F(r)
$$
is a non-decreasing function for $0<r<1/2$, and in particular for $0<s<r<1/2$ there holds
\begin{align*}
 W(r;v,f,(x^0,t^0))-W(s;v,f,(x^0,t^0))+F(r)-F(s)\\\geq \int_s^r\frac{1}{\tau^5}\int_{\R^n\times (-\tau^2+t^0,t^0]}\frac{(Lv)^2}{-t}G(x,-t) d t d x d \tau.
\end{align*}
\end{prop}

In view of the proposition above, the limit
$$W(0^+;u,1):=\lim_{r\to 0}W(r;u,1)=\lim_{r\to 0}W(r;u,f)$$ exists. 
In Lemma 6.2 and Lemma 6.3 in \cite{CPS04}, the possible values are determined. It turns out that they are in general $0$, $15/2$ and $15$. As is explained in Section 7.1 in \cite{CPS04} the value 
$0$, corresponding to so-called zero energy points, does not occur for $X^0\in \Gamma$.

\begin{definition}\label{reg} We say that the the point $(x^0,t^0)\in \Gamma$ is a low energy 
point if
$$  
W(0^+;u,f,(x^0,t^0)):=\lim_{r\to 0}W(r;v,f,(x^0,t^0))=\frac{15}{2}.
$$
\end{definition}
From Proposition \ref{prop:mono} it follows that the function
$$
(x,t)\mapsto W(0^+;u,f,(x,t)),
$$
is upper semi-continuous, and thus the set of regular points is an open set.
\subsection{The proof} We can now give the proof of the second main theorem.
\begin{proof}[~Proof of Theorem \ref{main2}] The proof consists of three steps. The first step amounts to prove that if the origin is a low energy point, then $u$ is non-negative close to the origin. The second step consists of applying the theory known for the case when $u$ has a sign, implying that the free boundary is locally a $C^1$ graph. These two steps prove the first part of the theorem.

In the third and final step, we observe that if the geometric condition holds at the origin, then Lemma 13.3 in \cite{CPS04} implies that the energy is sufficiently low (below the threshold 15) in a uniform neighbourhood of the origin. In particular, this implies that all points in that neighbourhood are low energy points. Hence, the first part of the theorem is applicable in a uniform neighborhood of the origin, which implies the second part of the theorem.\\
\noindent{\bf Step 1:  For $r_0$ small enough, $u\geq 0$ in $Q_{r^0}^-$}. We argue by contradiction. If this is not true, then there is a solution $u$ of \eqref{main}, with
$$\|u\|_{L^1(Q_1^-)}+\|f\|_{C^\text{Dini}}\leq M,\quad W(0;u,f,(0,0))=\frac{15}{2},$$
and points $(y^j,s^j)\in Q_1^-\to (0,0)$ so that $u(y^j,s^j)< 0$. 
Let $(x^j,t^j)$ be the closest free boundary point, i.e., 
$$
r_j=\dist(\Gamma,(y^j,s^j))=\dist(y^j-x^j,s^j-t^j)=\sqrt{|x^j-y^j|^2+|t^j-s^j|}.
$$
Define the rescaled functions
$$
v^j(x,t)=\frac{u(r_j x+x^j,r_j^2 t+t^j)}{r_j^2}.
$$
Clearly $v_j$ satisfies the equation
$$
H v_j = f(r_j x+x^j,r_j^2 t+t^j)\chi_{\{v_j\neq 0\}}, \text{ in $Q_\frac{1}{r_j}^-(-x^j,-t^j)$}.
$$
Moreover, due to Theorem \ref{main}
$$
\sup_{Q_\rho^-}|v_j|\leq C\rho^2, \text{ for $\rho <\frac{1}{2r_j}$},
$$
and by the choice of $(y^j,s^j)$ and $(x^j,t^j)$, $H v_j =f(r_j x+x^j,r_j^2 t+t^j)$ in the set 
$$\left\{(x,t): \Big|x-\frac{x^j-y^j}{r_j}\Big|^2+\Big|t-\frac{t^j-s^j}{r_j}\Big|< 1\right\}.$$
By standard estimates for parabolic equations, we can extract a sub-sequence, again labelled $v_j$, such that $v_j\to v_0$ uniformly and
$$
H v_0 = \chi_{\{v_0\neq 0\}}, \text{ in $\R^n\times\R^-$},\quad (0,0)\in \Gamma (v_0)
$$
$$
\sup_{Q_\rho^-}|v_0|\leq C\rho^2, \text{ for all $\rho >0$},
$$
and
\begin{equation}\label{h1}
H v_0 =1\text{ in $\{(x,t): |x-z^0|^2+|t-\tau^0|< 1\}$},\quad  v_0(z^0,\tau^0)\leq 0,
\end{equation}
where 
$$(z^0,\tau^0)=\lim_{j\to\infty}\left(\frac{y^j-x^j}{r_j},\frac{s_j-t_j}{r_j}\right),\quad (z^0)^2+\tau^0=1.$$
Observe that (\ref{h1}) assures that $v_0\not\equiv 0$. Moreover non-degeneracy 
(Lemma 5.1 in \cite{CPS04}) implies that the origin is contained in the free boundary
of $v_0$ (see also section 5.2 in \cite{CPS04}).

Next we need to use the assumption on the energy functional $W(r,u,f,X)$. Since
$W(r,u,f,X)$ is uniformly continuous in $X$ for each $r>0$ and 
$$\lim_{r\to 0}W(r,u,f,0)=\frac{15}{2},$$
the monotonicity formula implies that for each $\e>0$ there exists an $r_\e>0$ such that
$$
W(r,u,f,(x^j,t^j))\le W(r,u,f,0)+\e< \frac{15}{2}+2\e
$$ 
if $r<r_\e$ and $j$ is large enough ($j$ may depend on $r$).
In particular, a rescaling implies that
$$
W\left(\frac{r}{r_j},v_j, f(r_jx,r_j^2 t), 0\right)\le \frac{15}{2}+2\e
$$
if $r< r_\e$  is small enough and $j$ large enough. Passing to the limit $j\to \infty$
and using that $W$ is almost monotone in its first argument
we may conclude that for any $\e>0$
\begin{equation}\label{halfplanelim}
W\left(r,v_0, 1, 0\right)\le \frac{15}{2}+2\e.
\end{equation}
From (\ref{halfplanelim}) and the second part of Lemma 9.2 in \cite{CPS04} we can conclude that 
\begin{equation}\label{v0ishalfspace}
v_0=\frac{1}{2}(x\cdot e)_+^2,
\end{equation}
for some unit vector $e$. This is a contradiction to (\ref{h1}) since
$(z^0,\tau^0)$ is, by construction, a point in the set $Hv_0=1$ at unit distance 
from the free boundary so by (\ref{v0ishalfspace}) $v_0(z^0,\tau^0)=\frac{1}{2}$
which contradicts (\ref{h1}).
 
\noindent{\bf Step 2:  Apply the results from \cite{LM11}}. Now we are in the situation of Theorem 1.9 in \cite{LM11} if we consider $u$ to be defined only in $Q_{r_0}^-$, i.e., $u\geq 0$ in $Q_{r_0}^-$ and the origin is a low energy point. Hence, there is a small neighbourhood, which might depend on $u$ itself, where the free boundary is a $C^1$ graph (in the parabolic sense).

\noindent{\bf Step 3: The geometric condition implies a uniform energy condition.} Theorem \ref{main1} applied to $u$ and then Lemma 13.3 in \cite{CPS04} applied to $ u(rx,r^2t)/r^2$ implies that if $\sigma(r)$ is large enough, $r$ small enough (both depending on $\|u\|_{L^1(Q_1^-)}$ and $\|f\|_{C^\text{Dini}(Q_1^-)})$ and
$$
\frac{\textup{MD}\left(\{x:u(x,-r^2)=0\}\cap B_r\right)}{r}>\sigma(r),
$$
then
$$
W(\rho;u,f,X)<15-\e_0,
$$
for all $X\in Q_{r/2}^-$, $\rho<r$ small enough (depending on $\|u\|_{L^1(Q_1^-)}$ and $\|f\|_{C^\text{Dini}(Q_1^-)}$). In particular, all free boundary points in $Q_{r/2}^-$ are low energy points. Hence, we can apply Step 1 and Step 2 to conclude that in $Q_{r/2}^-$, the free boundary is a $C^1$ graph.

\section{Acknowledgements}
Henrik Shahgholian thanks the Swedish Research Council for financial support. Both John Andersson and Erik Lindgren are grateful to the the G\"oran Gustafsson foundation. Erik Lindgren also thanks the Royal Swedish Academy of Sciences for partial financial support.

\end{proof}

\bibliographystyle{plain}
\bibliography{ref.bib}

\end{document}